\documentclass[12pt,a4paper,reqno]{amsart} 

\usepackage[T1]{fontenc}       

\usepackage{amsfonts}          

\usepackage{amsmath}

\usepackage{amssymb}

\usepackage{overpic}
\usepackage{euscript} 
\usepackage{eurosym}
\usepackage{comment}
\usepackage{mathrsfs} 
\usepackage{ifthen}
\usepackage{esint} 
\usepackage{color}

\long\def\symbolfootnote[#1]#2{\begingroup%
\def\thefootnote{\fnsymbol{footnote}}\footnote[#1]{#2}\endgroup}

{\qed\vspace{5pt}}

\usepackage{amsthm}

\newtheoremstyle{lause}
{5pt}
{5pt}
{\slshape}
{\parindent}
{\bfseries}
{.}
{.5em}
{}

\theoremstyle{lause}

\newtheoremstyle{maaritelma}
{5pt}
{5pt}
{\rmfamily}
{\parindent}
{\bfseries}
{.}
{.5em}
{}

\theoremstyle{maaritelma}
\newtheoremstyle{lause}
{5pt}
{5pt}
{\slshape}
{\parindent}
{\bfseries}
{.}
{.5em}
{}

\theoremstyle{lause}
\newtheorem{theorem}{Theorem}[section]
\newtheorem{lemma}[theorem]{Lemma}

\newtheorem{corollary}[theorem]{Corollary}
\newtheorem{problem}[theorem]{Problem}

\newtheoremstyle{maaritelma}
{5pt}
{5pt}
{\rmfamily}
{\parindent}
{\bfseries}
{.}
{.5em}
{}

\theoremstyle{maaritelma}
\newtheorem{definition}[theorem]{Definition}

\newtheorem{example}[theorem]{Example}
\newtheorem{remark}[theorem]{Remark}

\numberwithin{equation}{section}

\DeclareMathOperator*{\essinf}{ess\,inf}

\begin{document}

\thispagestyle{empty}

\begin{center}

{\large{\textbf{Inner Riesz balayage in minimum energy problems\\ with external fields}}}

\vspace{18pt}

\textbf{Natalia Zorii}

\vspace{18pt}


\footnotesize{\address{Institute of Mathematics, Academy of Sciences
of Ukraine, Tereshchenkivska~3, 01601,
Kyiv-4, Ukraine\\
natalia.zorii@gmail.com }}

\end{center}

\vspace{12pt}

{\footnotesize{\textbf{Abstract.} For the Riesz kernel $\kappa_\alpha(x,y):=|x-y|^{\alpha-n}$ on $\mathbb R^n$, where $n\geqslant2$, $\alpha\in(0,2]$, and $\alpha<n$, we consider the problem of minimizing the Gauss functional
\[\int\kappa_\alpha(x,y)\,d(\mu\otimes\mu)(x,y)+2\int f\,d\mu,\quad\text{where $f:=-\int\kappa_\alpha(\cdot,y)\,d\omega(y)$},\]
$\omega$ being a given positive (Radon) measure on $\mathbb R^n$, and $\mu$ ranging over all positive measures of finite energy, concentrated on $A\subset\mathbb R^n$ and having unit total mass. We prove that if $A$ is a quasiclosed set of nonzero inner capacity $c_*(A)$, and if the inner balayage $\omega^A$ of $\omega$ onto $A$ is of finite energy, then the solution $\lambda_{A,f}$ to the problem in question exists if and only if either $c_*(A)<\infty$, or $\omega^A(\mathbb R^n)\geqslant1$. Despite its simple form, this result improves substantially some of the latest ones, e.g.\ those by Dragnev et al.\ (Constr.\ Approx., 2023) as well as those by the author (J.\ Math.\ Anal.\ Appl., 2023).
We also provide alternative characterizations of $\lambda_{A,f}$, and analyze its support. As an application, we show that if $A$ is not inner $\alpha$-thin at infinity, and is "very thin" at $z\in\partial_{\mathbb R^n}A$ (to be precise, if $c_*(A^*_z)<\infty$, $A^*_z$ being the inverse of $A$ with respect to the unit sphere centered at $z$), then the above problem with $\omega:=q\varepsilon_z$, where $q\in[1,\infty)$ and $\varepsilon_z$ denotes the unit Dirac measure at $z$, is still solvable. Thus no compensation effect occurs between the two oppositely signed charges, $-q\varepsilon_z$ and $\lambda_{A,f}$, carried by the same conductor $A$, which seems to contradict our physical intuition.
}}
\symbolfootnote[0]{\quad 2010 Mathematics Subject Classification: Primary 31C15.}
\symbolfootnote[0]{\quad Key words: Minimum Riesz energy problems with external fields; inner Riesz balayage; inner Riesz equilibrium measure; inner $\alpha$-harmonic measure; inner $\alpha$-thin\-n\-ess and $\alpha$-ultra\-thin\-n\-ess of a set at infinity.
}

\vspace{6pt}

\markboth{\emph{Natalia Zorii}} {\emph{Inner Riesz balayage in minimum energy problems with external fields}}

\section{Introduction and statement of the problem}\label{sec-alpha}

Fix $n\in\mathbb N$, $n\geqslant2$, and $\alpha\in(0,n)$, $\alpha\leqslant2$. The current work deals with minimum energy problems in the presence of external fields $f$ with respect to the Riesz kernel $\kappa_\alpha(x,y):=|x-y|^{\alpha-n}$ on $\mathbb R^n$ ($|x-y|$ being the Euclidean distance between $x$ and $y$), a point of interest for many researchers (see e.g.\ the monographs \cite{BHS,ST} and references therein, as well as \cite{CSW,Dr0}, \cite{Z-Rarx}--\cite{Z-pseudo}, some of the latest papers on this topic).

To formulate the problem, we first introduce some notations.
We denote by $\mathfrak M$ the linear space of all (re\-al-val\-u\-ed Radon) measures $\mu$ on $\mathbb R^n$, equipped with the {\it vague} topology of pointwise convergence on the class $C_0(\mathbb R^n)$ of all continuous functions $\varphi:\mathbb R^n\to\mathbb R$ of compact support, and by $\mathfrak M^+$ the cone of all positive $\mu\in\mathfrak M$, where $\mu$ is {\it positive} if and only if $\mu(\varphi)\geqslant0$ for all positive $\varphi\in C_0(\mathbb R^n)$.

Given $\mu,\nu\in\mathfrak M$, we define the {\it potential} $U^\mu$ and the {\it mutual energy} $I(\mu,\nu)$ by
\begin{align*}
 U^\mu(x)&:=\int\kappa_\alpha(x,y)\,d\mu(y),\quad x\in\mathbb R^n,\\
 I(\mu,\nu)&:=\int\kappa_\alpha(x,y)\,d(\mu\otimes\nu)(x,y),
\end{align*}
respectively, provided that the integral on the right is well defined (as a finite number or $\pm\infty$). For $\mu=\nu$, $I(\mu,\nu)$ defines the {\it energy} $I(\mu):=I(\mu,\mu)$ of $\mu\in\mathfrak M$.

When speaking of a positive measure $\mu\in\mathfrak M^+$, we understand that its potential $U^\mu$
is not identically infinite: $U^\mu\not\equiv+\infty$ on $\mathbb R^n$; or equivalently (cf.\ \cite[Section~I.3.7]{L})
\begin{equation*}
\int_{|y|>1}\,\frac{d\mu(y)}{|y|^{n-\alpha}}<\infty.
\end{equation*}
Actually, then (and only then) $U^\mu$ is finite {\it qua\-si-ev\-ery\-where} ({\it q.e.})\ on $\mathbb R^n$, cf.\
\cite[Section~III.1.1]{L}, namely everywhere except for a subset of outer capacity zero.\footnote{For {\it outer} and {\it inner} capacities, denoted by $c^*(\cdot)$ and $c_*(\cdot)$, respectively, see  \cite[Section~II.2.6]{L}. We write $c(Q):=c_*(Q)$ if $Q\subset\mathbb R^n$ is {\it capacitable} (e.g.\ Borel, see \cite[Theorem~2.8]{L})~--- that is, if $c_*(Q)=c_*(Q)$. A proposition $\mathcal P(x)$ involving a variable point $x\in\mathbb R^n$ is said to hold {\it nearly everywhere} ({\it n.e.})\ on $A\subset\mathbb R^n$ if the set of all $x\in A$ where $\mathcal P(x)$ fails, is of inner capacity zero.} This would necessarily hold if $\mu$ were required to be {\it bounded} (that is, with $\mu(\mathbb R^n)<\infty$), or of finite energy, cf.\ \cite[Corollary to Lemma~3.2.3]{F1}.

It was discovered by Riesz \cite[Chapter~I, Eq.~(13)]{R} (cf.\ \cite[Theorem~1.15]{L}) that the Riesz kernel $\kappa_\alpha$ is {\it strictly positive definite} in the sense that $I(\mu)\geqslant0$ for all (signed) $\mu\in\mathfrak M$, and moreover $I(\mu)=0\iff\mu=0$. Due to this fact, all $\mu\in\mathfrak M$ with $I(\mu)<\infty$ form a pre-Hil\-bert space $\mathcal E$ with the inner product $\langle\mu,\nu\rangle:=I(\mu,\nu)$ and the energy norm $\|\mu\|:=\sqrt{I(\mu)}$, see e.g.\ \cite[Lemma~3.1.2]{F1}. The topology on $\mathcal E$ defined by means of the norm $\|\cdot\|$, is said to be {\it strong}.

It is crucial to our study that the cone $\mathcal E^+:=\mathfrak M^+\cap\mathcal E$ is {\it complete} in the induced strong topology, and the strong topology on $\mathcal E^+$ is finer than the (induced) vague topology on $\mathcal E^+$ (see Deny
\cite{D1}; for $\alpha=2$, see also Cartan \cite{Ca1}). Thus any strong Cauchy sequence (net) $(\mu_j)\subset\mathcal E^+$
converges {\it both strongly and vaguely} to the same unique limit $\mu_0\in\mathcal E^+$, the strong topology on $\mathcal E$ as well as the vague topology on $\mathfrak M$ being Hausdorff. (Following Fuglede \cite{F1}, such a kernel is said to be {\it
perfect}.)

For any $A\subset\mathbb R^n$, we denote by $\mathfrak M^+(A)$ the class of all $\mu\in\mathfrak M^+$ {\it concentrated on} $A$, which means that $A^c:=\mathbb R^n\setminus A$ is $\mu$-neg\-lig\-ible, or equivalently that $A$ is $\mu$-mea\-s\-ur\-ab\-le and $\mu=\mu|_A$, $\mu|_A$ being the trace of $\mu$ to $A$, cf.\ \cite[Section~V.5.7]{B2}. (If $A$ is closed, then $\mu\in\mathfrak M^+(A)$ if and only if $S(\mu)\subset A$,
$S(\mu)$ being the support of $\mu$.) We define
\[\mathcal E^+(A):=\mathfrak M^+(A)\cap\mathcal E,\quad\breve{\mathcal E}^+(A):=\bigl\{\mu\in\mathcal E^+(A):\ \mu(\mathbb R^n)=1\bigr\}.\]

Fix a universally measurable function $f:\mathbb R^n\to[-\infty,\infty]$, treated as an {\it external field} acting on the measures (charges) carried by $\overline{A}:={\rm Cl}_{\mathbb R^n}A$, and let $\breve{\mathcal E}^+_f(A)$ be the class of all $\mu\in\breve{\mathcal E}^+(A)$ such that $f$ is $\mu$-in\-t\-eg\-r\-ab\-le (see \cite{B2}, Chapter~IV, Sections~3,~4), or equivalently with
\[I_f(\mu):=\|\mu\|^2+2\int f\,d\mu\in(-\infty,\infty).\]
Define
\[w_f(A):=\inf_{\mu\in\breve{\mathcal E}^+_f(A)}\,I_f(\mu)\in[-\infty,\infty],\]
where, as usual, the infimum over the empty set is interpreted as $+\infty$.

\begin{problem}\label{pr}If $-\infty<w_f(A)<\infty$, does there exist $\lambda_{A,f}\in\breve{\mathcal E}^+_f(A)$ with \[I_f(\lambda_{A,f})=w_f(A)?\]
\end{problem}

Being originated by Gauss \cite{Gau} (cf.\ also Frostman \cite[Section~17]{Fr}), Problem~\ref{pr} is nowadays often referred to as {\it the inner Gauss variational problem} \cite{O}, while $I_f(\cdot)$ is termed {\it the Gauss functional} \cite{L}. (In constructive function theory, $I_f(\cdot)$ is often said to be {\it the $f$-weighted energy}, see e.g.\ \cite{BHS,ST}.)

The solution $\lambda_{A,f}$ to Problem~\ref{pr} is {\it unique}. This follows easily from the convexity of the class $\breve{\mathcal E}^+_f(A)$ by making use of the parallelogram identity in the
pre-Hil\-bert space $\mathcal E$ and the strict positive definiteness of the Riesz kernel (see e.g.\ \cite[Lemma~6]{Z5a}).

Concerning the solvability of Problem~\ref{pr}, in the existing literature on this topic the set $A$ is usually required to be closed, while the external field $f$ to be lower semicontinuous (l.s.c.)\ on $A$ (see \cite{BHS,CSW,Dr0,L,O,ST}). If moreover $f\geqslant0$ unless $A$ is compact, then the Gauss functional $I_f(\cdot)$ becomes vaguely l.s.c.\ on $\mathfrak M^+(A)$, which plays a decisive role in the analysis of the solvability of Problem~\ref{pr}.

$\P$ However, the requirement of lower semicontinuity of $f$ on $\overline{A}$ is not, in general, fulfilled even if\footnote{Unless, of course, $S(\omega)\cap\overline{A}=\varnothing$.}
\begin{equation}\label{f}
f:=-U^\omega,\quad\text{where $\omega\in\mathfrak M^+$}.
\end{equation}

Henceforth, we assume (\ref{f}) to hold. To explore Problem~\ref{pr} for such external fields and for quite general (not necessarily closed) sets $A$, we apply the approach developed in our recent paper \cite{Z-Rarx}, which is based on the systematic use of both the strong and the vague topologies on the cone $\mathcal E^+$, related to each other as indicated above. Due to a refined application of that approach as well as of the theories of inner balayage and inner equilibrium measures,\footnote{Regarding the theory of inner Riesz balayage and that of inner Riesz equilibrium measures, see the author's recent papers \cite{Z-bal}--\cite{Z-arx-22}, cf.\ also Section~\ref{sec-baleq} below for a brief survey.} we have improved a number of the latest results on this topic, thereby discovering new interesting phenomena, which seem to be quite surprising and even to contradict our physical intuition.

To formulate the results thereby obtained (Section~\ref{sec-main}), we first recall the following theorem, providing characteristic properties of the minimizer $\lambda_{A,f}$; it can be derived from the author's earlier paper
\cite{Z5a} (see Theorems~1, 2 and Proposition~1 therein).

\begin{theorem}\label{th-ch2}For $\lambda\in\breve{\mathcal E}^+_f(A)$ to be the {\rm(}unique{\rm)} solution
$\lambda_{A,f}$ to Problem~{\rm\ref{pr}}, it is necessary and sufficient that either of the following two inequalities be
fulfilled:
\begin{align}U_f^\lambda&\geqslant\int U_f^\lambda\,d\lambda\quad\text{n.e.\ on $A$},\label{1}\\
U_f^\lambda&\leqslant w_f(A)-\int f\,d\lambda\quad\text{$\lambda$-a.e.\ on $\mathbb R^n$,}\label{2}
\end{align}
where
\[U_f^\lambda:=U^\lambda+f\]
is said to be the $f$-weighted potential of $\lambda$.
If either of {\rm(\ref{1})} or {\rm(\ref{2})} holds true, then equality actually prevails in {\rm(\ref{2})}, and moreover
\begin{equation}\label{cc}
\int U_f^{\lambda}\,d\lambda=w_f(A)-\int f\,d\lambda=:c_{A,f}\in(-\infty,\infty),
\end{equation}
$c_{A,f}$ being referred to as the inner $f$-weighted equilibrium constant.
\end{theorem}

\section{Main results}\label{sec-main}

$\bullet$ To avoid trivialities, we assume throughout this paper that
\begin{equation}\label{capnon0}
 c_*(A)>0.
\end{equation}
Then (and only then) $\mathcal E^+(A)$ is not reduced to $\{0\}$, cf.\ \cite[Lemma~2.3.1]{F1}, whence
\[\breve{\mathcal E}^+(A)\ne\varnothing.\]

$\bullet$ Unless explicitly stated otherwise, the following $(\mathcal H_1)$--$(\mathcal H_3)$ are required to hold:
\begin{itemize}
  \item[$(\mathcal H_1)$] The cone $\mathcal E^+(A)$ is closed in the strong topology on $\mathcal E^+$.
  \item[$(\mathcal H_2)$] The external field $f$ is given by means of formula (\ref{f}) with $\omega\in\mathfrak M^+$.
  \item[$(\mathcal H_3)$] $\omega^A$, the inner balayage of the above $\omega$ onto $A$, is of finite energy:
\begin{equation}\label{ob}
 \omega^A\in\mathcal E^+.
\end{equation}
\end{itemize}

Under these general conventions, which will usually not be repeated henceforth, we have the following useful Lemmas~\ref{lfin} and \ref{f-strcont}.

\begin{lemma}\label{lfin}The Gauss functional $I_f(\cdot)$ is representable in the form
\begin{equation}\label{If}
 I_f(\mu)=\|\mu-\omega^A\|^2-\|\omega^A\|^2\quad\text{for all $\mu\in\mathcal E^+(A)$}.
\end{equation}
Therefore,
\begin{equation}\label{Iff}\breve{\mathcal E}^+_f(A)=\breve{\mathcal E}^+(A),\end{equation}
hence
\begin{equation}\label{lf1'}
w_f(A)=\inf_{\mu\in\breve{\mathcal E}^+(A)}\,I_f(\mu),\end{equation}
and moreover
\begin{equation}\label{lf1}
-\infty<w_f(A)<\infty.
\end{equation}
\end{lemma}

\begin{proof} As $U^{\omega^A}=U^\omega$ n.e.\ on $A$ (see (\ref{ineq1}) with $\zeta:=\omega$), the same equality is fulfilled $\mu$-a.e.\ on $\mathbb R^n$ for all $\mu\in\mathcal E^+(A)$.\footnote{Here we have utilized the fact, to be often useful in what follows, that any $\mu$-mea\-s\-ur\-ab\-le set $E\subset\mathbb R^n$ with $c_*(E)=0$ is $\mu$-negligible for any $\mu\in\mathcal E^+$, that is, $\mu^*(E)=0$ (cf.\ \cite[Lemma~3.5]{Z-Rarx}).\label{f-negl}} Therefore, for any $\mu\in\mathcal E^+(A)$,
\begin{align*}
 &\int U^\omega\,d\mu=\int U^{\omega^A}\,d\mu=\langle\omega^A,\mu\rangle\in[0,\infty),\\
 &I_f(\mu)=\|\mu\|^2+2\int f\,d\mu=\|\mu\|^2-2\langle\omega^A,\mu\rangle,
\end{align*}
whence (\ref{If})--(\ref{lf1'}), and consequently, by the strict positive definiteness of $\kappa_\alpha$,
\[w_f(A)\geqslant-\|\omega^A\|^2>-\infty.\]
By \cite[Lemma~5]{Z5a}, the remaining claim $w_f(A)<\infty$ is equivalent to the inequality
\[c_*\bigl(\{x\in A:\ |f(x)|<\infty\}\bigr)>0,\]
which indeed holds true by virtue of (\ref{capnon0}) and the fact that $f$ is finite q.e., hence n.e.\ on $\mathbb R^n$. (Here we have used a strengthened
version of countable subadditivity for inner capacity, provided by Lemma~\ref{str-sub} below.)
\end{proof}

\begin{lemma}\label{str-sub}
 For arbitrary $Q\subset\mathbb R^n$ and universally measurable $U_j\subset\mathbb R^n$,
\[c_*\Bigl(\bigcup_{j\in\mathbb N}\,Q\cap U_j\Bigr)\leqslant\sum_{j\in\mathbb N}\,c_*(Q\cap U_j).\]
\end{lemma}

\begin{proof}
See \cite[pp.~157--158]{F1} (for $\alpha=2$, cf.\ \cite[p.~253]{Ca2}); compare with \cite[p.~144]{L}.
\end{proof}

\begin{lemma}\label{f-strcont} The cone $\mathcal E^+(A)$ is complete in the induced strong topology, while the Gauss functional $I_f(\cdot)$ is strongly continuous on $\mathcal E^+(A)$. In more detail, if a net $(\mu_s)_{s\in S}\subset\mathcal E^+(A)$ is strong Cauchy, then there exists the unique $\mu_0\in\mathcal E^+(A)$ such that $\mu_s\to\mu_0$ strongly and vaguely as $s$ ranges through $S$, and moreover
\[\lim_{s\in S}\,I_f(\mu_s)=I_f(\mu_0).\]
\end{lemma}

\begin{proof} Being a strongly closed subcone of the strongly complete cone $\mathcal E^+$, see $(\mathcal H_1)$, $\mathcal E^+(A)$ must be strongly complete. Therefore, a strong Cauchy net $(\mu_s)_{s\in S}\subset\mathcal E^+(A)$ must converge to some (unique) $\mu_0\in\mathcal E^+(A)$ strongly, hence also vaguely, the Riesz kernel being perfect (Section~\ref{sec-alpha}). Since the norm $\|\cdot\|$ is strongly continuous on $\mathcal E$, the strong continuity of $I_f(\cdot)$ on $\mathcal E^+(A)$ follows directly from representation (\ref{If}).
\end{proof}

\begin{remark}\label{f-quasi} As shown in \cite[Theorem~3.9]{Z-Rarx}, $(\mathcal H_1)$ is fulfilled, for instance, if $A$ is {\it quasiclosed} ({\it quasicompact}), that is, if $A$ can be approximated in outer capacity by closed (compact) sets (Fuglede  \cite[Definition~2.1]{F71}).\end{remark}

\begin{remark}\label{rem-pc}
Assumption $(\mathcal H_3)$ is satisfied, for example, if
\begin{equation}\label{pc}
\omega=\omega_0:=\tau+\sigma,
\end{equation}
where $\tau\in\mathcal E^+$, while $\sigma\in\mathfrak M^+$ is bounded and meets the separation condition
\begin{equation*}\inf_{(x,y)\in S(\sigma)\times A}\,|x-y|>0\end{equation*}
(see \cite[Section~4.6]{Z-Rarx}). For some other occurrences of (\ref{ob}) see Section~\ref{sec-ap} below.
\end{remark}

\subsection{On the solvability of Problem~\ref{pr}}\label{sec-solv} Due to (\ref{lf1}), Problem~\ref{pr} on the existence of $\lambda_{A,f}$, minimizing $I_f(\cdot)$ over $\breve{\mathcal E}^+_f(A)$ (equivalently, over $\breve{\mathcal E}^+(A)$~--- see (\ref{Iff})), makes sense, and moreover such a minimizer is unique (if it exists). Necessary and sufficient conditions for $\lambda_{A,f}$ to exist are provided by the following Theorem~\ref{th-main}.

\begin{theorem}\label{th-main}
For $\lambda_{A,f}$ to exist, it is necessary and sufficient that either
\begin{equation}\label{capf}
 c_*(A)<\infty,
\end{equation}
or
\begin{equation}\label{balineq}
\omega^A(\mathbb R^n)\geqslant1.
\end{equation}
Furthermore, if $\lambda_{A,f}$ exists, then
\begin{align}
c_{A,f}&<0\quad\text{if\/ $\omega^A(\mathbb R^n)>1$},\label{C2'}\\
c_{A,f}&=0\quad\text{if\/ $\omega^A(\mathbb R^n)=1$},\label{C1'}\\
c_{A,f}&>0\quad\text{if\/ $\omega^A(\mathbb R^n)<1$},\label{C3'}
\end{align}
$c_{A,f}$ being the inner $f$-weighted equilibrium constant {\rm(see Theorem~\ref{th-ch2})}.
\end{theorem}

\begin{corollary}\label{cornoth}
If $A$ is not inner $\alpha$-thin at infinity,\footnote{For the concept of inner $\alpha$-thinness of a set at infinity, see \cite{KM,Z-bal2} (cf.\ also Definition~\ref{def-thin} below).} then
\[\lambda_{A,f}\text{\ exists}\iff\omega(\mathbb R^n)\geqslant1.\]
\end{corollary}

\begin{proof}
  This follows directly from Theorem~\ref{th-main} by use of the fact that, if $A$ is not inner $\alpha$-thin at infinity, then $\mu^A(\mathbb R^n)=\mu(\mathbb R^n)$ for all $\mu\in\mathfrak M^+$, see Theorem~\ref{th-eq}(vi$_3$).
  \end{proof}

\begin{corollary}\label{cornoth''}If $A$ is quasicompact, then $\lambda_{A,f}$ does exist.\end{corollary}

\begin{proof}This is implied by Theorem~\ref{th-main}, for, being approximated in outer capacity by compact sets, a quasicompact set must be of finite outer (hence, inner) capacity.\end{proof}

\begin{corollary}\label{cornoth'}
If $c_*(A)=\infty$ and $\omega(\mathbb R^n)<1$, then $\lambda_{A,f}$ fails to exist.\end{corollary}

\begin{proof} We deduce this at once from Theorem~\ref{th-main} by noting that $\mu^A(\mathbb R^n)\leqslant\mu(\mathbb R^n)$ for all $\mu\in\mathfrak M^+$ (see Section~\ref{some}, (b) with $\zeta:=\mu$).\end{proof}

\subsection{Alternative characterizations of $\lambda_{A,f}$}\label{sec-alt} Assume $\lambda_{A,f}$  exists (see Section~\ref{sec-solv} for necessary and/or sufficient conditions for this to hold). Then, by (\ref{1}) and (\ref{cc}),
\begin{equation*}
 \lambda_{A,f}\in\Lambda_{A,f},
\end{equation*}
where
\begin{equation}\label{gamma}
\Lambda_{A,f}:=\bigl\{\mu\in\mathfrak M^+: \ U^\mu_f\geqslant c_{A,f}\quad\text{n.e.\ on\ $A$}\bigr\}.
\end{equation}

\begin{theorem}\label{th-main1} If moreover
\begin{equation}\label{eqq}
\omega^A(\mathbb R^n)\leqslant1,
\end{equation}
then $\lambda_{A,f}$ admits the representation
\begin{equation}\label{RRR}\lambda_{A,f}=\left\{
\begin{array}{cl}\omega^A+c_{A,f}\gamma_A&\text{if \ $c_*(A)<\infty$},\\
\omega^A&\text{otherwise},\\ \end{array} \right.
\end{equation}
where $\gamma_A$ denotes the inner equilibrium measure on $A$, normalized by $\gamma_A(\mathbb R^n)=c_*(A)$. This $\lambda_{A,f}$ can alternatively be characterized by any one of the following {\rm(i)}--{\rm(iii)}.
\begin{itemize}
\item[{\rm(i)}] $\lambda_{A,f}$ is the unique measure in the class $\Lambda_{A,f}$ of minimum potential, i.e.\footnote{This implies immediately that
$\lambda_{A,f}$ can also be characterized as the unique measure in the class $\Lambda_{A,f}$ of minimum $f$-we\-i\-g\-h\-t\-ed potential~--- now, however, qua\-si-ev\-ery\-where on $\mathbb R^n$:
\begin{equation*} U^{\lambda_{A,f}}_f=\min_{\mu\in\Lambda_{A,f}}\,U^\mu_f\quad\text{q.e.\ on\ $\mathbb R^n$}.\end{equation*}
This relation as well as (\ref{minpot}) and (\ref{minen}) would remain valid if $\mu$ in (\ref{gamma}) were of finite energy.}
\begin{equation}\label{minpot} U^{\lambda_{A,f}}=\min_{\mu\in\Lambda_{A,f}}\,U^\mu\quad\text{on\ $\mathbb R^n$}.\end{equation}
\item[{\rm(ii)}] $\lambda_{A,f}$ is the unique measure in the class $\Lambda_{A,f}$ of minimum energy, i.e.
\begin{equation}\label{minen}I(\lambda_{A,f})=\min_{\mu\in\Lambda_{A,f}}\,I(\mu).\end{equation}
\item[{\rm(iii)}] $\lambda_{A,f}$ is the unique measure in the class $\mathcal E^+(A)$ having the property
\begin{equation*}
U^{\lambda_{A,f}}_f=c_{A,f}\quad\text{n.e.\ on $A$}.
\end{equation*}
Furthermore,
\begin{align}\label{C1}
  c_{A,f}&\geqslant0\quad\text{if\/ $c_*(A)<\infty$},\\
  c_{A,f}&=0\quad\text{otherwise}.\label{C2}
\end{align}
\end{itemize}
\end{theorem}

\subsection{A description of the support of $\lambda_{A,f}$}\label{sec-descrrr} In the following Theorem~\ref{th-main3}, a set $A$ is assumed to be closed,\footnote{For closed $A$, $(\mathcal H_1)$ is necessarily fulfilled \cite[Theorem~3.9]{Z-Rarx}, and hence can be dropped.} and such that for every $x\in A$ and every neighborhood $U_x$ of $x$ in $\mathbb R^n$, we have $c_*(A\cap U_x)>0$; such $A$ is said to coincide with its {\it reduced kernel}, cf.\ \cite[p.~164]{L}. For the sake of simplicity of formulation, in Theorem~\ref{th-main3} we also assume that $A^c$ is connected unless $\alpha<2$. Denote $\partial A:=\partial_{\mathbb R^n}A$.

\begin{theorem}\label{th-main3}Under the hypotheses of Theorem~{\rm\ref{th-main1}}, assume moreover that
\[\omega\ne\omega|_{A_R}\quad\text{and}\quad\omega|_{A_R}\in\mathcal E^+,\] where $A_R$ denotes the set of all $\alpha$-reg\-ul\-ar points of $A$. Then
\begin{equation}\label{RRR'}S(\lambda_{A,f})=\left\{
\begin{array}{cl}A&\text{if \ $\alpha<2$},\\
S(\omega|_{A_R})\cup\partial A&\text{otherwise}.\\ \end{array} \right.
\end{equation}
\end{theorem}

Regarding the concept of $\alpha$-reg\-ul\-ar points for Borel $A$, see e.g.\ \cite[Section~V.1.2]{L} (cf.\ \cite[Section~6]{Z-bal}, where $A$ is arbitrary; see also Section~\ref{some} below). Also note that $A_R$ is Borel measurable \cite[Theorem~5.2]{Z-bal2}, and hence the trace $\omega|_{A_R}$ is well defined.

\begin{remark}
Let $n\geqslant3$, $A:=\overline{\mathbb R^{n-1}}$,
$\mathbb R^{n-1}$ being imbedded in $\mathbb R^n$, and let $f:=-U^{\varepsilon_x}$, where $\varepsilon_x$ is the unit Dirac measure at $x\in\mathbb R^n\setminus\overline{\mathbb R^{n-1}}$. As stated in \cite[Section~4.2, case~(i)]{Dr0}, $S(\lambda_{A,f})=\overline{\mathbb R^{n-1}}$. This result is correct and agrees with (\ref{RRR'}), for ${\rm Int}_{\mathbb R^n}\overline{\mathbb R^{n-1}}=\varnothing$, and hence our Remark~2.21 in \cite{Z-Rarx} is false.
\end{remark}

In applications, it is often useful to know whether for unbounded $A$, $S(\lambda_{A,f})$ might be compact. Some conditions ensuring this, can be derived from Theorem~\ref{th-main3}.
See also  Theorem~\ref{th-main2} and Corollary~\ref{th-main2'}, cf.\ Theorem~\ref{shrpp} and Remarks~\ref{shap}, \ref{shap'}.

\begin{theorem}\label{th-main2}Assume that $\omega^A(\mathbb R^n)>1$,\footnote{According to Theorem~\ref{th-main}, $\lambda_{A,f}$ then does exist, and moreover $c_{A,f}<0$. The same is applicable to Corollary~\ref{th-main2'}, which is again seen from Theorem~\ref{th-main}, now by making use of Theorem~\ref{th-eq}(vi$_3$).} and that
\begin{equation}\label{L}
\lim_{|x|\to\infty, \ x\in A}\,U^\omega(x)=:L\text{ \ exists}.
\end{equation}
If moreover $L=0$,\footnote{This obviously holds true if $\omega$ is of compact support.} then $\lambda_{A,f}$ is of compact support.
\end{theorem}

\begin{remark}\label{shap}Theorem~\ref{th-main2} is sharp in the sense that it would fail in general if $\omega^A(\mathbb R^n)>1$ were replaced by $\omega^A(\mathbb R^n)\leqslant1$. Indeed, take $A$ as indicated at the beginning of this subsection, and suppose additionally that $A$ is unbounded, whereas $c_*(A)<\infty$ (such $A$ does exist, see \cite[Section~V.2.8]{L}). Then for $\omega:=q\varepsilon_x$, where $x\in A^c$ and $q\in(0,1]$, $\lambda_{A,f}$ does exist (Theorem~\ref{th-main}). Furthermore, $\omega^A(\mathbb R^n)\leqslant\omega(\mathbb R^n)\leqslant1$ (see Section~\ref{some}, (b) with $\zeta:=\omega$), and hence $S(\lambda_{A,f})$ is noncompact (Theorem~\ref{th-main3}).\end{remark}

\begin{corollary}\label{th-main2'}Assume that $\omega(\mathbb R^n)>1$, and that {\rm(\ref{L})} holds. Then $\lambda_{A,f}$ is of compact support whenever $A$ is not inner $\alpha$-thin at infinity.\end{corollary}

\begin{remark}\label{shap'}Corollary~\ref{th-main2'} is sharp in the sense that it would fail in general if either $\omega(\mathbb R^n)=1$, or if $A$ were inner $\alpha$-thin at infinity. For the former claim, take $A$ as indicated at the beginning of this subsection, and suppose that $\partial A$ is unbounded. If $\omega(\mathbb R^n)=1$, while $A$ is not $\alpha$-thin at infinity, then $\omega^A(\mathbb R^n)=1$ (Theorem~\ref{th-eq}(vi$_3$)), whence $\lambda_{A,f}$ does exist (Theorem~\ref{th-main}), but has noncompact support (Theorem~\ref{th-main3}). For the latter claim, see the following Theorem~\ref{shrpp}.

\begin{theorem}\label{shrpp}
Let $A$ be as indicated at the beginning of this subsection, and let $A$ be $\alpha$-thin at infinity. Then there are $x_j\in A^c$ and $q_j\in[1,\infty)$, $j\in\mathbb N$, such that
\[\lim_{j\to\infty}\,|x_j|=\lim_{j\to\infty}\,q_j=\infty,\]
and moreover all the supports $S(\lambda_{A,f_j})$, where $f_j:=-q_jU^{\varepsilon_{x_j}}$, are noncompact.
\end{theorem}
\end{remark}

\subsection{On the novelty of the above results}\label{novelty} All the results of this section are largely new. However, if $\omega=\omega_0$ is of form (\ref{pc}), then, under some additional requirements on the external field $-U^{\omega_0}$, Theorem~\ref{th-main1} and Corollaries~\ref{cornoth}--\ref{cornoth'}, \ref{th-main2'} were given in \cite{Z-Rarx}.
Also, if $A$ is closed and not $\alpha$-thin at infinity, while $\omega_0$ in (\ref{pc}) is just $q\varepsilon_x$, where $q\in(0,\infty)$ and $x\in A^c$, then Corollaries~\ref{cornoth}, \ref{cornoth'}, \ref{th-main2'}
were obtained in \cite{Dr0}.

Proofs of the above results are presented in Sections~\ref{sec-prep}, \ref{sec-proofs}. For the convenience of the reader, in
Section~\ref{sec-baleq} we review some facts of the theories of inner balayage and inner equilibrium measures, the summary being organized correspondingly to its usage in this work. Finally, in Section~\ref{sec-ap} we apply the obtained results to  $\omega:=q\varepsilon_z$, $z\in\partial A$ being suitably chosen, thereby discovering new interesting phenomena.

\section{Preliminaries}\label{sec-baleq}

Unless explicitly stated otherwise, in this section we do not require $(\mathcal H_1)$ to hold.

\subsection{On the inner balayage}\label{some} The theory of inner (Riesz) balayage (sweeping out) of arbitrary $\zeta\in\mathfrak M^+$ to arbitrary $A\subset\mathbb R^n$, $n\geqslant2$, generalizing Cartan's theory \cite{Ca2} of inner Newtonian balayage ($\alpha=2$, $n\geqslant3$) to any $\alpha\in(0,2]$, $\alpha<n$, was originated in the author's recent papers \cite{Z-bal,Z-bal2}, and it found a further development in \cite{Z-arx1}--\cite{Z-arx-22}.\footnote{For the theory of {\it outer} Riesz balayage, see e.g.\ the monographs \cite{BH,Br,Doob}, the last two dealing with the Newtonian kernel ($\alpha=2$, $n\geqslant3$).}

Denote
\begin{equation}\label{gammaa}
   \Gamma_{A,\zeta}:=\bigl\{\mu\in\mathfrak M^+:\ U^\mu\geqslant U^\zeta\quad\text{n.e.\ on $A$}\bigr\},
  \end{equation}
and let $\mathcal E'(A)$ stand for the closure of $\mathcal E^+(A)$ in the strong topology on $\mathcal E^+$. Being a strongly closed subcone of the strongly complete cone $\mathcal E^+$, $\mathcal E'(A)$ is likewise strongly complete, and it is convex.

\begin{theorem}\label{th-bal1}
For any $\zeta:=\sigma\in\mathcal E^+$, there is precisely one $\sigma^A\in\mathcal E'(A)$, called the inner balayage of $\sigma$ to $A$, that is determined by any one of the following {\rm(i$_1$)}--{\rm(iii$_1$)}.
\begin{itemize}
  \item[{\rm(i$_1$)}] There exists the unique $\sigma^A\in\mathcal E'(A)$ having the property\footnote{That is, the inner balayage $\sigma^A$ of $\sigma\in\mathcal E^+$ to $A$ is, actually, the orthogonal projection of $\sigma$ in the pre-Hil\-bert space $\mathcal E$ onto the (convex, strongly complete) cone $\mathcal E'(A)$, cf.\ \cite[Theorem~1.12.3]{E2}.}
  \[\|\sigma-\sigma^A\|=\min_{\mu\in\mathcal E'(A)}\,\|\sigma-\mu\|.\]
  \item[{\rm(ii$_1$)}] There exists the unique $\sigma^A\in\mathcal E'(A)$ satisfying the equality
  \begin{equation*}
   U^{\sigma^A}=U^\sigma\quad\text{n.e.\ on $A$}.
  \end{equation*}
  \item[{\rm(iii$_1$)}] There exists the unique $\sigma^A\in\Gamma_{A,\sigma}$ of minimum energy in the class $\Gamma_{A,\sigma}$, $\Gamma_{A,\sigma}$ being introduced by {\rm(\ref{gammaa})} with $\zeta:=\sigma$. That is,
  \[I(\sigma^A)=\min_{\mu\in\Gamma_{A,\sigma}}\,I(\mu).\]
 \end{itemize}
\end{theorem}

\begin{proof}
  See \cite[Section~3]{Z-bal}, \cite[Section~4]{Z-arx1}, and \cite[Section~3]{Z-arx-22}.
\end{proof}

\begin{remark}\label{r-h} If $(\mathcal H_1)$ holds (in particular, if $A$ is quasiclosed), then obviously
$\mathcal E'(A)=\mathcal E^+(A)$,
and so (i$_1$) and (ii$_1$) in Theorem~\ref{th-bal1} remain valid with $\mathcal E^+(A)$ in place of $\mathcal E'(A)$. Thus the inner balayage $\sigma^A$, where $\sigma\in\mathcal E^+$, is then necessarily concentrated on $A$, and it is uniquely determined within $\mathcal E^+(A)$ by $U^{\sigma^A}=U^\sigma$ n.e.\ on $A$.\footnote{This characteristic property of $\sigma^A$ would obviously fail to hold if $A$ were arbitrary.}
\end{remark}

\begin{theorem}\label{th-bal2}For any $\zeta\in\mathfrak M^+$, there exists precisely one $\zeta^A\in\mathfrak M^+$, called the inner balayage of $\zeta$ to $A$, that is determined by any one of the following {\rm(i$_2$)}--{\rm(iii$_2$)}.
\begin{itemize}
\item[{\rm(i$_2$)}] There exists the unique $\zeta^A\in\Gamma_{A,\zeta}$ of minimum potential in $\Gamma_{A,\zeta}$, that is,
  \[U^{\zeta^A}=\min_{\mu\in\Gamma_{A,\zeta}}\,U^\mu\quad\text{on $\mathbb R^n$}.\]
  \item[{\rm(ii$_2$)}] There exists the unique $\zeta^A\in\mathfrak M^+$ satisfying the symmetry relation\footnote{Relation (\ref{eq-sym}) can actually be extended to $\sigma$ of infinite energy, that is (see \cite[Corollary~4.1]{Z-bal}),
\[I(\zeta^A,\mu)=I(\zeta,\mu^A)\quad\text{for all $\zeta,\mu\in\mathfrak M^+$}.\]}
 \begin{equation}\label{eq-sym}
    I(\zeta^A,\sigma)=I(\zeta,\sigma^A)\quad\text{for all $\sigma\in\mathcal E^+$},
  \end{equation}
  where $\sigma^A$ is uniquely determined by Theorem~{\rm\ref{th-bal1}}.
   \item[{\rm(iii$_2$)}] There exists the unique $\zeta^A\in\mathfrak M^+$ satisfying either of the two limit relations
  \begin{align*}\sigma_j^A&\to\zeta^A\quad\text{vaguely in $\mathfrak M^+$ as $j\to\infty$},\\
U^{\sigma_j^A}&\uparrow U^{\zeta^A}\quad\text{pointwise on $\mathbb R^n$ as $j\to\infty$},
\end{align*}
where $(\sigma_j)\subset\mathcal E^+$ denotes an arbitrary sequence having the property\footnote{Such a sequence $(\sigma_j)\subset\mathcal E^+$ does exist (see e.g.\ \cite[p.~272]{L} or \cite[p.~257, footnote]{Ca2}).}
\[U^{\sigma_j}\uparrow U^\zeta\quad\text{pointwise on $\mathbb R^n$ as $j\to\infty$},\]
while $\sigma_j^A$ is uniquely determined by Theorem~{\rm\ref{th-bal1}}.
\end{itemize}
\end{theorem}

\begin{proof}
  See \cite[Sections~3, 4]{Z-bal}.
\end{proof}

\begin{remark}\label{rem-nodet}
  It is worth noting that, although for $\zeta\in\mathfrak M^+$, we still have
  \begin{equation}\label{ineq1}
   U^{\zeta^A}=U^\zeta\quad\text{n.e.\ on $A$},
  \end{equation}
  see \cite[Theorem~3.10]{Z-bal}, this equality no longer determines $\zeta^A$ uniquely (as it does for $\zeta:=\sigma\in\mathcal E^+$, cf.\ Theorem~\ref{th-bal1}(ii$_1$) or Remark~\ref{r-h}), which can be seen by taking $\zeta:=\varepsilon_y$, $y$ being an inner $\alpha$-irregular point for $A$ (for definition see below).\end{remark}

A point $y\in\overline{A}$ is said to be {\it inner $\alpha$-irregular} for $A$ if $\varepsilon_y^A\ne\varepsilon_y$; we denote by $A_I$ the set of all those $y$. By the Wiener type criterion \cite[Theorem~6.4]{Z-bal}, $A_I$ consists of all $y\in\overline{A}$ such that
\begin{equation}\label{w}\sum_{j\in\mathbb N}\,\frac{c_*(A_j)}{r^{j(n-\alpha)}}<\infty,\end{equation}
where $r\in(0,1)$ and $A_j:=A\cap\{x\in\mathbb R^n:\ r^{j+1}<|x-y|\leqslant r^j\}$ (and hence $A_I\subset\partial A$); while by the Kel\-logg--Ev\-ans type theorem \cite[Theorem~6.6]{Z-bal},\footnote{Observe that both (\ref{w}) and (\ref{KE}) refer to inner capacity; compare with the Kel\-l\-ogg--Ev\-ans and Wiener type theorems established for {\it outer} balayage (see e.g.\ \cite{BH,Br,Ca2,Doob}). Regarding (\ref{KE}), also note that the whole set $A_I$ may be of nonzero capacity \cite[Section~V.4.12]{L}.}
\begin{equation}\label{KE}c_*(A\cap A_I)=0.\end{equation}
All other points of $\overline{A}$ are said to be {\it inner $\alpha$-regular} for $A$; we denote $A_R:=\overline{A}\setminus A_I$. Alternatively (see \cite[Lemma~6.3]{Z-bal}),
\[y\in A_R\iff U^{\mu^A}(y)=U^\mu(y)\quad\text{for all $\mu\in\mathfrak M^+$}.\]

Along with (\ref{ineq1}), the following properties of the inner balayage $\zeta^A$, $\zeta\in\mathfrak M^+$ and $A\subset\mathbb R^n$ being arbitrary, will be useful in the sequel.

\begin{itemize}
\item[(a)] $U^{\zeta^A}\leqslant U^\zeta$ everywhere on $\mathbb R^n$ (see \cite[Theorem~3.10]{Z-bal}).
\item[(b)] Principle of positivity of mass: $\zeta^A(\mathbb R^n)\leqslant\zeta(\mathbb R^n)$ (see \cite[Corollary~4.9]{Z-bal}).
\item[(c)] Balayage "with a rest": $\zeta^A=(\zeta^Q)^A$ for any $Q\supset A$ (see \cite[Corollary~4.2]{Z-bal}).
\item[(d)] If $c_*(A)>0$, then $\zeta^A=0\iff\zeta=0$.
\item[(e)] For any $\zeta_1,\zeta_2\in\mathfrak M^+$ and $a_1,a_2\in[0,\infty)$, $(a_1\zeta_1+a_2\zeta_2)^A=a_1\zeta_1^A+a_2\zeta_2^A$.
\end{itemize}

\subsection{On the inner equilibrium measure}\label{sec-eq}
Throughout the present paper, the inner (Riesz) equilibrium measure $\gamma_A$ of $A\subset\mathbb R^n$ is understood in an extended sense where $I(\gamma_A)$ as well as $\gamma_A(\mathbb R^n)$ might be $+\infty$.
Define
\begin{equation*}
\Gamma_A:=\bigl\{\mu\in\mathfrak M^+:\ U^\mu\geqslant1\quad\text{n.e.\ on $A$}\bigr\}.\end{equation*}

\begin{definition}[{\rm see \cite[Section~5]{Z-bal}}]\label{def-eq} $\gamma_A\in\mathfrak M^+$ is said to be {\it the inner equilibrium measure} of $A\subset\mathbb R^n$ if it is of minimum potential in $\Gamma_A$, that is, if $\gamma_A\in\Gamma_A$ and
\begin{equation}\label{G'}U^{\gamma_A}=\min_{\mu\in\Gamma_A}\,U^\mu\quad\text{on $\mathbb R^n$}.\end{equation}
\end{definition}

Such $\gamma_A$ is unique (if it exists), which is obvious from \cite[Theorem~1.12]{L}.

As seen from the following theorem, the concept of inner equilibrium measure is closely related to that of inner balayage. For any $y\in\mathbb R^n$, we denote by $A_y^*$ the inverse of $A$ with respect to $S_{y,1}:=\{x\in\mathbb R^n:\ |x-y|=1\}$, cf.\ \cite[Section~IV.5.19]{L}.

\begin{theorem}\label{th-eq} For arbitrary $A\subset\mathbb R^n$, the following {\rm(i$_3$)}--{\rm(vi$_3$)} are equivalent.
\begin{itemize}
  \item[{\rm(i$_3$)}] There exists the {\rm(}unique{\rm)} inner equilibrium measure $\gamma_A$ of $A$.
  \item[{\rm(ii$_3$)}] There exists $\nu\in\mathfrak M^+$ having the property
\[\essinf_{x\in A}\,U^\nu(x)>0,\]
the infimum being taken over all of $A$ except for a subset of $c_*(\cdot)=0$.
\item[{\rm(iii$_3$)}]For some {\rm(}equivalently, every{\rm)} $y\in\mathbb R^n$,
\begin{equation}\label{iii}\sum_{j\in\mathbb N}\,\frac{c_*(A_j)}{r^{j(n-\alpha)}}<\infty,\end{equation}
where $r\in(1,\infty)$ and $A_j:=A\cap\{x\in\mathbb R^n:\ r^j\leqslant|x-y|<r^{j+1}\}$.
\item[{\rm(iv$_3$)}]For some {\rm(}equivalently, every{\rm)} $y\in\mathbb R^n$,
\begin{equation*}y\not\in(A_y^*)_R.\end{equation*}
\item[{\rm(v$_3$)}]For some {\rm(}equivalently, every{\rm)} $y\in\mathbb R^n$, $\varepsilon_y^{A_y^*}$ is $C$-absolutely continuous.\footnote{$\mu\in\mathfrak M^+$ is said to be {\it $C$-absolutely continuous} if $\mu(K)=0$ for every compact set $K\subset\mathbb R^n$ with $c(K)=0$. This certainly occurs if $I(\mu)<\infty$, but not conversely (see \cite[pp.~134--135]{L}).\label{f-C}}
\item[{\rm(vi$_3$)}] There exists $\chi\in\mathfrak M^+$ having the property\footnote{Compare with Section~\ref{some}, (b).}
\[\chi^A(\mathbb R^n)<\chi(\mathbb R^n).\]
    \end{itemize}

Furthermore, if any one of these {\rm(i$_3$)--(vi$_3$)} is fulfilled, then for every $y\in\mathbb R^n$,
\begin{equation}\label{har-eq}\varepsilon_y^{A_y^*}=\gamma_A^*,\end{equation}
where $\gamma_A^*$ denotes the Kelvin transform of $\gamma_{A}$ with respect to the sphere $S_{y,1}$.\footnote{For the concept of Kelvin transformation, see \cite[Section~14]{R} as well as \cite[Section~IV.5.19]{L}.}
\end{theorem}

\begin{proof}
See \cite{Z-bal2} (Theorem~2.1 and Corollary~5.3).
\end{proof}

\begin{definition}[{\rm cf.\ \cite[Definition~2.1]{Z-bal2}}]\label{def-thin} $A\subset\mathbb R^n$ is said to be {\it inner $\alpha$-thin at infinity} if any one of the above {\rm(i$_3$)--(vi$_3$)} holds true.
\end{definition}

\begin{remark}\label{doob}
As seen from (\ref{iii}), the concept of inner $\alpha$-thinness of a set at infinity thus defined coincides with that introduced by Kurokawa and Mizuta \cite{KM}. If $\alpha=2$ while a set in question is Borel, then this concept also coincides with that of {\it outer} thinness of a set at infinity by Doob \cite[pp.~175--176]{Doob}.
\end{remark}

Given $A\subset\mathbb R^n$, we denote by $\mathfrak C_A$ the upward directed set of all compact subsets $K$ of $A$, where $K_1\leqslant K_2$ if and only if $K_1\subset K_2$. If a net $(x_K)_{K\in\mathfrak C_A}\subset Y$ converges to $x_0\in Y$, $Y$ being a topological space, then we shall indicate this fact by writing
\begin{equation*}x_K\to x_0\text{ \ in $Y$ as $K\uparrow A$}.\end{equation*}

Assume that the inner equilibrium measure $\gamma_A$ exists (or equivalently that $A$ is inner $\alpha$-thin at infinity). By use of properties of the Kelvin transformation (see \cite[Section~IV.5.19]{L}), we derive from Theorem~\ref{th-eq} and the facts reviewed in Section~\ref{some} that $\gamma_A$ is $C$-absolutely continuous, supported by $\overline{A}$, and having the property
\begin{equation}
0<U^{\gamma_A}\leqslant1\quad\text{on $\mathbb R^n$}.\label{ineqne}
\end{equation}
In addition,\footnote{However, the equality $U^{\gamma_A}=1$ n.e.\ on $A$ does not characterize $\gamma_A$ uniquely, cf.\ Theorem~\ref{th-eq''}.}
\begin{equation}
U^{\gamma_A}=1\quad\text{on $A_R$\quad(hence, n.e.\ on $A$)}.\label{eqne}
\end{equation}
The same $\gamma_A$ can be uniquely determined by either of the two limit relations
\begin{align}
  \gamma_K&\to\gamma_A\quad\text{vaguely in $\mathfrak M^+$ as $K\uparrow A$},\label{cv1}\\
  U^{\gamma_K}&\uparrow U^{\gamma_A}\quad\text{pointwise on $\mathbb R^n$ as $K\uparrow A$},\label{cv2}
\end{align}
where $\gamma_K$ denotes the only measure in $\mathcal E^+(K)$ with $U^{\gamma_K}=1$ n.e.\ on $K$~--- namely, the (classical) equilibrium measure on $K$ \cite[Section~II.1.3]{L}, normalized by
\begin{equation}\label{k}
\gamma_K(\mathbb R^n)=\|\gamma_K\|^2=c(K).
\end{equation}
(For the alternative characterizations (\ref{cv1}) and (\ref{cv2}) of $\gamma_A$, see \cite[Lemma~5.3]{Z-bal}.)

If moreover $c_*(A)<\infty$, then (and only then) the above $\gamma_A$ is of finite energy,\footnote{In fact, the "if" part of this claim is obtained from (\ref{cv1}) and (\ref{k}) by use of \cite[Eq.~(1.4.5)]{L} (the principle of descent), whereas the "only if" part follows from the chain of inequalities
\[c(K)=\int U^{\gamma_A}\,d\gamma_K=\int U^{\gamma_K}\,d\gamma_A\leqslant\int U^{\gamma_A}\,d\gamma_A=I(\gamma_A)<\infty\quad\text{for all $K\in\mathfrak C_A$},\]
where the first equality is implied by (\ref{eqne}) and (\ref{k}), while the first inequality~--- by (\ref{cv2}).}
and it can alternatively be characterized as the only measure in $\Gamma_A\cap\mathcal E$ of minimum energy (see \cite[Theorems~6.1, 9.2]{Z-arx-22}, cf.\ \cite[Theorem~2.6]{L}):
\begin{equation}\label{cv4'}\|\gamma_A\|^2=\min_{\mu\in\Gamma_A\cap\mathcal E}\,\|\mu\|^2.\end{equation}
Furthermore, then, along with (\ref{cv1}) and (\ref{cv2}),
\begin{equation*}
\gamma_K\to\gamma_A\quad\text{strongly in $\mathcal E^+$ as $K\uparrow A$}
\end{equation*}
(see \cite[Theorem~8.1]{Z-arx-22}), and therefore
\[\gamma_A\in\mathcal E'(A).\]

On account of \cite[Theorem~7.2(c)]{Z-arx-22}, we are thus led to the following conclusion.

\begin{theorem}\label{th-eq''}Assume $c_*(A)<\infty$ and $(\mathcal H_1)$ holds. Then the inner equilibrium measure $\gamma_A$, uniquely determined by either of {\rm(\ref{G'})} or {\rm(\ref{cv4'})}, belongs to $\mathcal E^+(A)$, i.e.
\begin{equation}\label{eqfin}
\gamma_A\in\mathcal E^+(A),
\end{equation}
and it is characterized as the only measure in $\mathcal E^+(A)$ with $U^{\gamma_A}=1$ n.e.\ on $A$.\end{theorem}

\begin{corollary}\label{cor-eq}
 Under the assumptions of Theorem~{\rm\ref{th-eq''}},
 \begin{equation}\label{eqbal}
 (\gamma_A)^A=\gamma_A.
\end{equation}
\end{corollary}

\begin{proof} As observed in Remark~\ref{r-h}, under the stated requirements $(\gamma_A)^A$ is concentrated on $A$, and it is uniquely determined within $\mathcal E^+(A)$ by $U^{(\gamma_A)^A}=U^{\gamma_A}$ n.e.\ on $A$.
Since $U^{\gamma_A}=1$ n.e.\ on $A$, see (\ref{eqne}), applying Lemma~\ref{str-sub} therefore gives $U^{(\gamma_A)^A}=1$ n.e.\ on $A$,
which according to Theorem~\ref{th-eq''} necessarily results in (\ref{eqbal}).
\end{proof}

\section{Proofs of the main assertions: preparatory results}\label{sec-prep}

In Sections~\ref{sec-prep} and \ref{sec-proofs}, a set $A\subset\mathbb R^n$, a measure $\omega\in\mathfrak M^+$, and an external field $f$ are assumed to satisfy $(\mathcal H_1)$--$(\mathcal H_3)$ and (\ref{capnon0}). Then, by $(\mathcal H_1)$,
\begin{equation}\label{ee}
\mathcal E'(A)=\mathcal E^+(A).
\end{equation}

\subsection{A dual extremal problem} Along with  $f=-U^\omega$, where $\omega\in\mathfrak M^+$ is such that $\omega^A\in\mathcal E^+$ (see $(\mathcal H_2)$ and $(\mathcal H_3)$), we consider the (dual) external field $\tilde{f}$ defined by
\begin{equation}\label{eq-dual1}\tilde{f}:=-U^{\omega^A}.\end{equation}
Then, by virtue of (\ref{ineq1}) with $\zeta:=\omega$,
\[\tilde{f}=f\quad\text{n.e.\ on $A$},\] hence $\mu$-a.e.\ for all $\mu\in\mathcal E^+(A)$ (footnote~\ref{f-negl}), which results in the following conclusion.

\begin{lemma}\label{l-dual} It holds true that
\begin{equation}\label{eq-dual}I_f(\mu)=\|\mu\|^2+2\int\tilde{f}\,d\mu=:I_{\tilde{f}}(\mu)\quad\text{for all $\mu\in\mathcal E^+(A)$},\end{equation}
hence
\begin{equation}\label{eq-dual3}w_{f}(A)=\inf_{\mu\in\breve{\mathcal E}^+(A)}\,I_{\tilde{f}}(\mu)=:w_{\tilde{f}}(A),\end{equation}
and consequently $\lambda_{A,f}$ exists if and only if so does $\lambda_{A,\tilde{f}}$, the solution to Problem~{\rm\ref{pr}} with $\tilde{f}$ in place of $f$, and in the affirmative case these two measures coincide:
\[\lambda_{A,f}=\lambda_{A,\tilde{f}}.\]
\end{lemma}

Problem~\ref{pr} with $\tilde{f}$ in place of $f$ may therefore be referred to as the {\it dual\/} problem. Since $\omega^A\in\mathcal E^+$, the dual problem can be treated by applying \cite[Sections~2.2--2.4]{Z-Rarx} with $\delta:=\tau:=\omega^A$, $\delta$ and $\tau$ being the measures appearing in \cite[Eq.~(2.8)]{Z-Rarx}.

In view of this observation, some of the proofs in the present research can be reduced to those in \cite{Z-Rarx}, performed for the (dual) external field $\tilde{f}=-U^{\omega^A}$.

\subsection{An auxiliary extremal problem} Theorem~\ref{th-bal1}(i$_1$) admits the following useful generalization.

\begin{theorem}\label{l-oo'}
The inner balayage $\omega^A$ is concentrated on $A$, i.e.\
$\omega^A\in\mathcal E^+(A)$, and it can actually be found as the unique solution to the problem of minimizing the Gauss functional $I_f(\mu)$, or equivalently $I_{\tilde{f}}(\mu)$, where $\mu$ ranges over $\mathcal E^+(A)$. Alternatively, $\omega^A$ is uniquely characterized
within $\mathcal E^+(A)$ by the equality $U^{\omega^A}=U^\omega$ n.e.\ on $A$.
\end{theorem}

\begin{proof} Since $\omega^A=(\omega^A)^A$ (see Section~\ref{some}, (c) with $Q:=A$), substituting (\ref{ee}) into Theorem~\ref{th-bal1}(i$_1$) with
$\sigma:=\omega^A\in\mathcal E^+$ shows that, indeed, $\omega^A\in\mathcal E^+(A)$ as well as that $\omega^A$ is the orthogonal projection of itself onto $\mathcal E^+(A)$; or equivalently that $\omega^A$ is the (unique) solution to the problem of minimizing  $\|\mu\|^2-2\int U^{\omega^A}\,d\mu$, $\mu$ ranging over $\mathcal E^+(A)$. In view of (\ref{eq-dual1}) and (\ref{eq-dual}), this proves the former part of the theorem.

For the latter part, assume that the equality $U^{\omega^A}=U^\omega$ n.e.\ on $A$ is also fulfilled for some $\mu_0\in\mathcal E^+(A)$ in place of $\omega^A$. Then, applying Lemma~\ref{str-sub} yields
\[U^{\mu_0}=U^\omega=U^{\omega^A}\quad\text{n.e.\ on $A$},\]
whence $\mu_0=(\omega^A)^A$, by use of Theorem~\ref{th-bal1}(ii$_1$) with $\sigma:=\omega^A\in\mathcal E^+$. Combining this with   $(\omega^A)^A=\omega^A$ (see above) gives $\mu_0=\omega^A$, thereby completing
the whole proof.\end{proof}

$\P$ According to Theorem~\ref{l-oo'}, $\omega^A$ is, in fact, the (unique) solution to the (auxiliary) problem of minimizing the Gauss functional $I_f(\mu)$ over the class $\mathcal E^+(A)$. That is,
\begin{gather}
\label{g1}\omega^A\in\mathcal E^+(A),\\
\label{hatw}
I_f(\omega^A)=\min_{\mu\in\mathcal E^+(A)}\,I_f(\mu)=:\widehat{w}_f(A).
\end{gather}

Returning to Problem~\ref{pr}, we remark that, since $\breve{\mathcal E}^+(A)\subset\mathcal E^+(A)$,
\begin{equation}\label{hatww}\widehat{w}_f(A)\leqslant w_f(A).\end{equation}
The case where equality prevails in (\ref{hatww}), will be of particular interest.

\begin{theorem}\label{equality}
Assume that
\begin{equation}\label{mass1}
\omega^A(\mathbb R^n)=1.\end{equation}
Then the solution $\lambda_{A,f}$ to Problem~{\rm\ref{pr}} does exist, and moreover
\begin{equation}\label{hatww'}\lambda_{A,f}=\omega^A,\quad w_f(A)=\widehat{w}_f(A),\quad c_{A,f}=0.\end{equation}
\end{theorem}

\begin{proof}
Noting from (\ref{g1}) and (\ref{mass1}) that $\omega^A\in\breve{\mathcal E}^+(A)$, we conclude by making use of (\ref{lf1'}), (\ref{hatw}), and (\ref{hatww}) that
\[\widehat{w}_f(A)\leqslant w_f(A)\leqslant I_f(\omega^A)=\widehat{w}_f(A),\]
which implies the existence of $\lambda_{A,f}$ as well as the first two equalities in (\ref{hatww'}). Since $\lambda_{A,f}=\omega^A$, applying (\ref{ineq1}) to $\omega$ gives $U_f^{\lambda_{A,f}}=U^{\omega^A}-U^\omega=0$
n.e.\ on $A$, hence $\lambda_{A,f}$-a.e. Therefore, in view of (\ref{cc}),
$c_{A,f}=\int U_f^{\lambda_{A,f}}\,d\lambda_{A,f}=0$, and the proof is complete.
\end{proof}

\subsection{Extremal measures} Taking into account that $\breve{\mathcal E}^+_f(A)=\breve{\mathcal E}^+(A)$ (Lemma~\ref{lfin}), we call a net $(\mu_s)_{s\in S}$ {\it minimizing} in Problem~\ref{pr} if $(\mu_s)_{s\in S}\subset\breve{\mathcal E}^+(A)$ and
\begin{equation}\label{min}
\lim_{s\in S}\,I_f(\mu_s)=w_f(A);
\end{equation}
let $\mathbb M_f(A)$ denote the class of all those $(\mu_s)_{s\in S}$. Since $w_f(A)$ is finite (see (\ref{lf1})),
\[\mathbb M_f(A)\ne\varnothing.\]

\begin{lemma}\label{l-extr}
  There exists the unique $\xi_{A,f}\in\mathcal E^+(A)$, called the extremal measure in Problem~{\rm\ref{pr}}, and such that for every minimizing net $(\mu_s)_{s\in S}\in\mathbb M_f(A)$,
  \begin{equation}\label{m-extr}
    \mu_s\to\xi_{A,f}\quad\text{strongly and vaguely in $\mathcal E^+(A)$}.
  \end{equation}
This yields
\begin{equation}\label{ext-eq1}
I_f(\xi_{A,f})=w_f(A).
\end{equation}
 \end{lemma}

\begin{proof} Relation (\ref{m-extr}) follows by standard arguments (cf.\ \cite[Proof of Lemma~4.1]{Z-Rarx}), based on the convexity of the class $\breve{\mathcal E}^+(A)$, a pre-Hil\-bert structure on the space $\mathcal E$, the perfectness of the Riesz kernel, and the strong completeness of the cone $\mathcal E^+(A)$ (see Lemma~\ref{f-strcont}). Since the Gauss functional $I_f(\cdot)$ is strongly continuous on $\mathcal E^+(A)$ (Lemma~\ref{f-strcont}), (\ref{ext-eq1}) is obtained by combining (\ref{min}) and (\ref{m-extr}).

Alternatively, this can be derived from \cite{Z-Rarx} (Lemmas~4.1 and 4.8, applied to $\tilde{f}$) by noting from (\ref{eq-dual3}) and (\ref{min}) that $\mathbb M_f(A)=\mathbb M_{\tilde{f}}(A)$, whence $\xi_{A,f}:=\xi_{A,\tilde{f}}$.\end{proof}

By use of the fact that the mapping $\mu\mapsto\mu(\mathbb R^n)$ is vaguely l.s.c.\ on $\mathfrak M^+$ \cite[Section~IV.1, Proposition~4]{B2}, we infer from (\ref{m-extr}) that,
in general,
\begin{equation}\label{ext-eq4}
\xi_{A,f}(\mathbb R^n)\leqslant1.
\end{equation}

\begin{corollary}\label{l-extr5}$\lambda_{A,f}$ exists if and only if equality prevails in {\rm(\ref{ext-eq4})}, i.e.
\begin{equation}\label{ext-eq2}
\xi_{A,f}(\mathbb R^n)=1,
\end{equation}
and in the affirmative case
\begin{equation}\label{ext-eq3}
\xi_{A,f}=\lambda_{A,f}.
\end{equation}
\end{corollary}

\begin{proof}
If (\ref{ext-eq2}) holds, then combining it with $\xi_{A,f}\in\mathcal E^+(A)$ and (\ref{ext-eq1}) gives (\ref{ext-eq3}). For the opposite, assume $\lambda_{A,f}$ exists. Since the trivial sequence $(\lambda_{A,f})$ is obviously minimizing, it must converge strongly to both $\lambda_{A,f}$ and $\xi_{A,f}$ (Lemma~\ref{l-extr}), which immediately results in (\ref{ext-eq3}), the strong topology on $\mathcal E$ being Hausdorff.

Alternatively, this can be deduced from \cite{Z-Rarx} (see the latter part of Lemma~4.8, applied to $\tilde{f}$) on account of the relations $\xi_{A,f}=\xi_{A,\tilde{f}}$ and $\lambda_{A,f}=\lambda_{A,\tilde{f}}$ (see above).
\end{proof}

\begin{lemma}\label{l-pot}
 For the extremal measure $\xi:=\xi_{A,f}$, we have
 \begin{equation}\label{e-pot1}
 U^\xi_f\geqslant C_\xi\quad\text{n.e.\ on $A$},
 \end{equation}
  where
 \begin{equation}\label{Cxi}
  C_\xi:=\int U^\xi_f\,d\xi\in(-\infty,\infty).
 \end{equation}
\end{lemma}

\begin{proof} Since $\xi_{A,f}=\xi_{A,\tilde{f}}$ while $f=\tilde{f}$ n.e.\ on $A$, we have
$U_f^{\xi_{A,f}}=U_{\tilde{f}}^{\xi_{A,\tilde{f}}}$ n.e.\ on $A$,
hence $\xi_{A,f}$-a.e., $\xi_{A,f}$ being of the class $\mathcal E^+(A)$ (Lemma~\ref{l-extr}). Therefore,
\[\int U_f^{\xi_{A,f}}\,d\xi_{A,f}=\int U_{\tilde{f}}^{\xi_{A,\tilde{f}}}\,d\xi_{A,\tilde{f}},\]
and (\ref{e-pot1}), (\ref{Cxi}) follow by use of \cite[Eqs.~(4.28), (4.29)]{Z-Rarx}, applied to $\tilde{f}$.

Regarding a direct proof of this lemma, its scheme is the following.
First, for each $K\in\mathfrak C_A$ with $c(K)>0$, there is the solution $\lambda_{K,f}$ to Problem~\ref{pr} with $A:=K$, for $\xi_{K,f}(\mathbb R^n)=1$, the class of all probability measures carried by $K$ being vaguely compact \cite[Section~III.1.9, Corollary~3]{B2}, whence $\xi_{K,f}=:\lambda_{K,f}$ (Corollary~\ref{l-extr5}).

Next, those $\lambda_{K,f}$ form a minimizing net, i.e.\ $(\lambda_{K,f})_{K\in\mathfrak C_A}\in\mathbb M_f(A)$, for \[w_f(K)\downarrow w_f(A)\quad\text{as $K\uparrow A$},\]
which in turn follows in a manner similar to that in the proof of Eq.~(4.10) in \cite{Z-Rarx}, by applying \cite[Lemma~1.2.2]{F1} to each of the positive, l.s.c., $\mu$-int\-eg\-r\-ab\-le functions $\kappa_\alpha$ and $U^\omega$, $\mu\in\breve{\mathcal E}^+(A)$ being arbitrary. (As for the $\mu$-integrability of $U^\omega$, see Lemma~\ref{lfin}.)

Thus $\lambda_{K,f}\to\xi_{A,f}$ strongly as $K\uparrow A$ (Lemma~\ref{l-extr}). The strong topology on $\mathcal E$ being first-coun\-t\-ab\-le, there exists a subsequence $(K_j)$ of the net $(K)_{K\in\mathfrak C_A}$ such that
\begin{equation*}U^{\lambda_{K_j,f}}_f\to U^{\xi_{A,f}}_f\quad\text{pointwise n.e.\ on $\mathbb R^n$ as $j\to\infty$,}\end{equation*}
cf.\ \cite[p.~166, Remark]{F1}. Now, applying (\ref{1}) to each of those $\lambda_{K_j,f}$, and then passing to the limits in the inequalities thereby obtained, we arrive at the claim, for
\begin{align*}\lim_{j\to\infty}\,c_{K_j,f}&=\lim_{j\to\infty}\,\Bigl(\|\lambda_{K_j,f}\|^2-\int U^\omega\,d\lambda_{K_j,f}\Bigr)=\lim_{j\to\infty}\,\bigl(\|\lambda_{K_j,f}\|^2-\langle\omega^A,\lambda_{K_j,f}\rangle\bigr)\\
{}&=\|\xi_{A,f}\|^2-\langle\omega^A,\xi_{A,f}\rangle=\|\xi_{A,f}\|^2-\int U^\omega\,d\xi_{A,f}=\int U^{\xi_{A,f}}_f\,d\xi_{A,f}.\end{align*}
(While doing that, we have again utilized the strengthened
version of countable subadditivity for inner capacity, presented by Lemma~\ref{str-sub}.)
 \end{proof}

\section{Proofs of the main assertions}\label{sec-proofs}

\subsection{Proof of Theorem~\ref{th-main} (the beginning)}\label{subs1} We first note that, if $\omega^A(\mathbb R^n)=1$, then the solution $\lambda_{A,f}$ does indeed exist, and moreover (\ref{C1'}) holds (see Theorem~\ref{equality}).

We shall next show that (\ref{capf}) is another sufficient condition for the existence of $\lambda_{A,f}$. According to Corollary~\ref{l-extr5}, this will follow once we prove that the extremal measure $\xi:=\xi_{A,f}$ has unit total mass, or equivalently (cf.\ Lemma~\ref{l-extr})
\begin{equation}\label{extun}
 \xi\in\breve{\mathcal E}^+(A).
\end{equation}
But the cone $\breve{\mathcal E}^+(A)$ is strongly closed, which is seen from $(\mathcal H_1)$ and (\ref{capf}) by applying   \cite[Theorem~3.6]{Z-Rarx}.  Since $\xi$ is the strong limit of a minimizing net $(\mu_s)_{s\in S}\in\mathbb M_f(A)$ (Lemma~\ref{l-extr}), whereas $(\mu_s)_{s\in S}\subset\breve{\mathcal E}^+(A)$, this gives (\ref{extun}), whence the claim.\footnote{A slight modification of these arguments shows that $c_*(A)<\infty$
still guarantees the solvability of Problem~\ref{pr} even if the external field $f$ is replaced by that of the form $f+u$, where $f$ is as above, while $u:\overline{A}\to(-\infty,\infty]$ is l.s.c., bounded from below, and such that $c_*(\{x\in A:\ u(x)<\infty\})>0$.}

To complete the proof of the "if" part, it remains to verify that $\lambda_{A,f}$ does exist if
\begin{equation}\label{strin3}
\omega^A(\mathbb R^n)>1.
\end{equation}
Assume first that $C_\xi\geqslant0$, $C_\xi$ being the (finite) number introduced by (\ref{Cxi}). Then, by virtue of (\ref{e-pot1}),
\[U^\xi\geqslant U^\omega+C_\xi\geqslant U^\omega\quad\text{n.e.\ on $A$},\]
which combined with $U^\omega=U^{\omega^A}$ n.e.\ on $A$ implies, by use of Lemma~\ref{str-sub}, that
\begin{equation}\label{innn}
U^\xi\geqslant U^{\omega^A}\quad\text{n.e.\ on $A$}.
\end{equation}
Noting that the measures $\omega^A$ and $\xi$ belong to $\mathcal E^+(A)$ (see Theorem~\ref{l-oo'} and Lemma~\ref{l-extr}, respectively), we infer from (\ref{innn}) by applying the principle of positivity of mass in the form stated in \cite[Theorem~1.5]{Z-Deny} that\footnote{Here we have used the fact that $\omega^A,\xi$ are $C$-absolutely continuous, being of finite energy.}
\[\xi(\mathbb R^n)\geqslant\omega^A(\mathbb R^n),\]
which in view of (\ref{ext-eq4}) contradicts (\ref{strin3}). We thus necessarily have
\begin{equation}\label{neg}
 C_\xi<0.
\end{equation}

Since (\ref{e-pot1}) holds true $\xi$-a.e., integrating it with respect to $\xi$ gives
\[C_\xi=\int U^\xi_f\,d\xi\geqslant C_\xi\cdot\xi(\mathbb R^n),\]
the equality being valid by (\ref{Cxi}). On account of (\ref{neg}), this implies
$\xi(\mathbb R^n)\geqslant1$, which together with (\ref{ext-eq4}) shows that, in fact,
$\xi(\mathbb R^n)=1$.
According to Corollary~\ref{l-extr5}, this proves the claimed solvability of Problem~\ref{pr} (with $\xi$ serving as the solution $\lambda_{A,f}$). Substituting $\xi=\lambda_{A,f}$ into (\ref{Cxi}), and then comparing the result obtained with (\ref{cc}), we also get $C_\xi=c_{A,f}$, which together with (\ref{neg}) establishes (\ref{C2'}).

Our next aim is to verify that Problem~\ref{pr} is unsolvable whenever\footnote{We provide here a direct proof of this claim, although it can actually be derived from \cite[Theorem~2.7]{Z-Rarx} with $\delta:=\tau:=\omega^A$ by use of the fact that $(\omega^A)^A=\omega^A$ (Section~\ref{some}, (c) with $Q:=A$).}
\begin{equation}\label{ii''}c_*(A)=\infty\quad\text{and}\quad\omega^A(\mathbb R^n)<1.\end{equation}
Due to the former of these assumptions, $c_*(A\setminus K)=\infty$ for any compact $K\subset\mathbb R^n$, and hence there exists a sequence $(\tau_j)\subset\breve{\mathcal E}^+(A)$ with
\begin{equation}\label{ii'''}\|\tau_j\|<1/j,\quad S(\tau_j)\subset\{|x|\geqslant j\}.\end{equation}
Define
\begin{equation}\label{muj}
\mu_j:=\omega^A+q\tau_j,\quad\text{where $q:=1-\omega^A(\mathbb R^n)\in(0,1)$},
\end{equation}
cf.\ the inequality in (\ref{ii''}). Then obviously $(\mu_j)\subset\breve{\mathcal E}^+(A)$. To show that, actually,
\begin{equation}\label{mujj}(\mu_j)\in\mathbb M_f(A),\end{equation}
we deduce from (\ref{muj}) with the aid of a straightforward verification that
\begin{align*}
 w_f(A)&\leqslant I_f(\mu_j)=I_f(\omega^A)+q^2\|\tau_j\|^2+2q\int\bigl(U^{\omega^A}-U^\omega\bigr)\,d\tau_j\\
 {}&=I_f(\omega^A)+q^2\|\tau_j\|^2=\widehat{w}_f(A)+q^2\|\tau_j\|^2\leqslant w_f(A)+q^2\|\tau_j\|^2,
\end{align*}
where the second equality holds since $U^{\omega^A}$ equals $U^\omega$ n.e.\ on $A$, hence $\tau_j$-a.e., while the subsequent two relations are due to (\ref{hatw}) and (\ref{hatww}). Letting now $j\to\infty$, in view of the former relation in (\ref{ii'''}) we obtain $\lim_{j\to\infty}I_f(\mu_j)=w_f(A)$, whence (\ref{mujj}).

Noting that $(\tau_j)$ is obviously vaguely convergent to zero, we have, by (\ref{muj}),
\begin{equation}\label{mujjj}\mu_j\to\omega^A\quad\text{vaguely as $j\to\infty$}.\end{equation}
The vague topology on $\mathfrak M$ being Hausdorff, we infer from (\ref{mujj}) and (\ref{mujjj}), by use of Lemma~\ref{l-extr}, that, actually, $\omega^A=\xi$, and hence $\xi(\mathbb R^n)<1$, by virtue of (\ref{ii''}). Applying Corollary~\ref{l-extr5} therefore yields that, in case (\ref{ii''}), Problem~\ref{pr} is indeed unsolvable.

The proof of Theorem~\ref{th-main} will be finalized in Section~\ref{subs3} by verifying (\ref{C3'}).

\subsection{Proof of Theorem~\ref{th-main1}} Under the stated assumptions, the solution $\lambda_{A,\tilde{f}}$ to the dual problem does exist, and moreover $\lambda_{A,\tilde{f}}=\lambda_{A,f}$ (Lemma~\ref{l-dual}). Since obviously
\begin{equation*}
U_{\tilde{f}}^{\lambda_{A,\tilde{f}}}=U_f^{\lambda_{A,f}}\quad\text{n.e.\ on $A$},
\end{equation*}
we get by integration $\int U_{\tilde{f}}^{\lambda_{A,\tilde{f}}}\,d\lambda_{A,\tilde{f}}=\int U_f^{\lambda_{A,f}}\,d\lambda_{A,f}$,
or equivalently
\begin{equation}\label{dual-c}c_{A,\tilde{f}}=c_{A,f},\end{equation}
$c_{A,\tilde{f}}$ being the inner $\tilde{f}$-weighted equilibrium constant for $A$ (see Theorem~\ref{th-ch2}). Thus
\begin{equation}\label{dual-L}\Lambda_{A,\tilde{f}}=\Lambda_{A,f},\end{equation}
the class $\Lambda_{A,\tilde{f}}$ being defined by (\ref{gamma}) with $f$ replaced by $\tilde{f}$.
Now, taking into account relation $(\omega^A)^A=\omega^A$ (see Section~\ref{some}, (c) with $Q:=A$) as well as (\ref{dual-c}) and (\ref{dual-L}), and applying  \cite[Theorem~2.7]{Z-Rarx} to $\lambda_{A,\tilde{f}}$, we obtain what was needed to prove.

However, in view of a decisive role of Theorem~\ref{th-main1} in the subsequent analysis, we would also like to provide its direct proof. As noted at the beginning of Section~\ref{sec-alt},
\begin{equation}\label{lambda}
 \lambda_{A,f}\in\Lambda_{A,f}.
\end{equation}

Suppose first that $c_*(A)=\infty$. Then (\ref{balineq}) must be fulfilled (Theorem~\ref{th-main}), which combined with (\ref{eqq}) yields
$\omega^A(\mathbb R^n)=1$. Applying Theorem~\ref{equality} we therefore get the latter relation in (\ref{RRR}) as well as (\ref{C2}). Substituting (\ref{C2}) into (\ref{gamma}) shows that in the case in question, $\Lambda_{A,f}$ consists, in fact, of all $\mu\in\mathfrak M^+$ having the property $U^\mu\geqslant U^\omega$ n.e.\ on $A$, or equivalently $U^\mu\geqslant U^{\omega^A}$ n.e.\ on $A$. That is,
\begin{equation}\label{d1}
\Lambda_{A,f}=\Gamma_{A,\omega^A},
\end{equation}
$\Gamma_{A,\omega^A}$ being defined by means of (\ref{gammaa}) with $\zeta:=\omega^A\in\mathcal E^+$.
Now, combining the latter equality in (\ref{RRR}) with $\omega^A=(\omega^A)^A$, we obtain
\begin{equation}\label{d2}
\lambda_{A,f}=(\omega^A)^A.
\end{equation}
On account of (\ref{lambda})--(\ref{d2}), applying Theorem~\ref{th-bal1}(iii$_1$) and Theorem~\ref{th-bal2}(i$_2$) to $\omega^A\in\mathcal E^+$ gives (ii) and (i), respectively. Finally, as seen from (\ref{C2}), (\ref{ee}), and (\ref{d2}), the remaining assertion (iii) will be proved once we show that $(\omega^A)^A$ is the only measure in $\mathcal E'(A)$ with $U^{(\omega^A)^A}=U^\omega$ n.e.\ on $A$, or equivalently $U^{(\omega^A)^A}=U^{\omega^A}$ n.e.\ on $A$, which, however, follows directly from Theorem~\ref{th-bal1}(ii$_1$) applied to $\sigma:=\omega^A\in\mathcal E^+$.

It is left to consider the case where $c_*(A)<\infty$. Define
\begin{equation}\label{delta}
\vartheta:=\omega^A+q\gamma_A,
\end{equation}
where $\gamma_A$ is the inner equilibrium measure on $A$ (see Theorem~\ref{th-eq''}), while
\[q:=\frac{1-\omega^A(\mathbb R^n)}{c_*(A)}.\]
As $\gamma_A(\mathbb R^n)=c_*(A)$, we infer from (\ref{eqq}), (\ref{eqfin}), and (\ref{g1}) that $q\geqslant0$ as well as
\begin{equation}
 \label{delta1}\vartheta\in\breve{\mathcal E}^+(A).
\end{equation}
Furthermore, it follows from (\ref{delta}), by use of (\ref{ineq1}), (\ref{eqne}), and Lemma~\ref{str-sub}, that
\[U_f^\vartheta=\bigl(U^{\omega^A}-U^\omega\bigr)+qU^{\gamma_A}=q\]
holds true n.e.\ on $A$, hence $\vartheta$-a.e., and consequently
\begin{equation}\label{delta2}U_f^\vartheta=\int U_f^\vartheta\,d\vartheta=q\quad\text{n.e.\ on $A$}.\end{equation}
By virtue of Theorem~\ref{th-ch2}, relations (\ref{delta1}) and (\ref{delta2}) show that, actually,
\begin{equation}\label{q''}q=c_{A,f}\quad\text{and}\quad\lambda_{A,f}=\vartheta=\omega^A+c_{A,f}\gamma_A,
\end{equation}
thereby establishing (\ref{C1}) as well as the former equality in (\ref{RRR}).

But $\gamma_A=(\gamma_A)^A$ (Corollary~\ref{cor-eq}) and $\omega^A=(\omega^A)^A$ (see above); whence, by (\ref{q''}),
\begin{equation}\label{delta''}
\lambda_{A,f}=(\omega^A+c_{A,f}\gamma_A)^A=\lambda_{A,f}^A.
\end{equation}
As seen from (\ref{delta2}) and (\ref{q''}), $U_f^{\lambda_{A,f}}=c_{A,f}$ n.e.\ on $A$, or equivalently \[U^{\lambda_{A,f}}=c_{A,f}+U^\omega\quad\text{n.e.\ on $A$}.\] Moreover, this equality uniquely determines $\lambda_{A,f}$ within $\mathcal E^+(A)$ $\bigl({}=\mathcal E'(A)\bigr)$, which follows from (\ref{delta''}) by virtue of Theorem~\ref{th-bal1}(ii$_1$) with $\sigma:=\lambda_{A,f}$. This proves (iii).

Using (\ref{ineq1}), (\ref{eqne}), and Lemma~\ref{str-sub} once again, we also observe that in the case in question, the class $\Lambda_{A,f}$ consists, in fact, of all $\mu\in\mathfrak M^+$ with
\[U^\mu\geqslant U^{\omega^A+c_{A,f}\gamma_A}=U^{\lambda_{A,f}}\quad\text{n.e.\ on $A$},\]
the equality being valid by virtue of (\ref{q''}). Thus
\[\Lambda_{A,f}=\Gamma_{A,\lambda_{A,f}},\]
which together with (\ref{delta''}) imply the remaining assertions (i) and (ii) by making use of Theorem~\ref{th-bal2}(i$_2$) and Theorem~\ref{th-bal1}(iii$_1$), respectively, with $\zeta:=\sigma:=\lambda_{A,f}$.

\subsection{Proof of Theorem~\ref{th-main} (the end)}\label{subs3} To complete the proof of Theorem~\ref{th-main}, it is left to verify (\ref{C3'}). We thus assume $\omega^A(\mathbb R^n)<1$; then necessarily $c_*(A)<\infty$, for if not, $\lambda_{A,f}$ would not exist (see Section~\ref{subs1} for a proof). Furthermore, $\omega^A(\mathbb R^n)<1$ necessarily yields $\lambda_{A,f}\ne\omega^A$, which together with the former equality in (\ref{RRR}) gives $c_{A,f}\ne0$. Combining this with (\ref{C1}) results in (\ref{C3'}).

\subsection{Proof of Theorem~\ref{th-main3}}\label{sec-pr-th-main3} Assume the requirements of the theorem are fulfilled. Denoting $\omega_1:=\omega|_{A_R^c}$,\footnote{$A_R$ being Borel measurable \cite[Theorem~5.2]{Z-bal2}, the trace $\omega|_{A_R}$ as well as $\omega|_{A_R^c}$ is well defined.} where $A_R^c:=(A_R)^c$, we have $\omega_1\in\mathfrak M^+(A_R^c)$, and hence
\begin{equation}\label{Dm}S(\omega_1^A)=\left\{
\begin{array}{cl}A&\text{if \ $\alpha<2$},\\
\partial A&\text{otherwise}.\\ \end{array} \right.
\end{equation}
Indeed, if $\omega_1:=\varepsilon_y$, where $y\in A_R^c$, then (\ref{Dm}) was established in \cite[Theorem~4.1]{Z-bal2}, while otherwise, it follows by making use of the integral representation formula
\[\zeta^A=\int\varepsilon_y^A\,d\zeta(y),\quad\zeta\in\mathfrak M^+\] (see \cite[Theorem~5.1]{Z-bal2}), applied to $\zeta:=\omega_1$.

Also note that since $\omega_2:=\omega|_{A_R}\in\mathcal E^+$ by assumption, $\omega_2\in\mathcal E^+(A)$, whence \begin{equation}\label{Dm'}\omega_2^A=\omega_2,\end{equation}
$\omega_2^A$ being the orthogonal projection of $\omega_2$ onto the strongly complete, convex cone $\mathcal E^+(A)$ (see Theorem~\ref{th-bal1}(i$_1$), combined with (\ref{ee})).

Assuming first that $c_*(A)=\infty$, on account of (\ref{RRR}) we have $\lambda_{A,f}=\omega^A$. Therefore, (\ref{Dm}) and (\ref{Dm'}) result in (\ref{RRR'}), for $\omega^A=\omega_1^A+\omega_2^A$ (see Section~\ref{some}, (e)).

If now $c_*(A)<\infty$, then $\lambda_{A,f}=\omega^A+c_{A,f}\gamma_A$, where $c_{A,f}\geqslant0$ (see Theorem~\ref{th-main1}), and (\ref{RRR'}) is obtained by combining (\ref{Dm}) and (\ref{Dm'}) with \cite[Theorem~7.2]{Z-bal}, providing a description of the support of the equilibrium measure $\gamma_A$.

\subsection{Proof of Theorem~\ref{th-main2}}\label{sec-prpr} Due to $\omega^A(\mathbb R^n)>1$, the solution $\lambda_{A,f}$ to Problem~\ref{pr} does exist, and moreover $c_{A,f}<0$ (Theorem~\ref{th-main}), whereas according to Theorem~\ref{th-ch2},
\begin{equation}
\label{ext-pr2}
U^{\lambda_{A,f}}_f=c_{A,f}\quad\text{$\lambda_{A,f}$-a.e.\ on $\mathbb R^n$.}
\end{equation}

Assume to the contrary that $S(\lambda_{A,f})$ is noncompact. Then, by (\ref{ext-pr2}),
there exists a sequence $(x_j)\subset A$ such that $|x_j|\to\infty$ as $j\to\infty$, and moreover
\begin{equation*}-c_{A,f}\leqslant U^{\lambda_{A,f}}(x_j)-c_{A,f}=U^\omega(x_j)\quad\text{for all $j\in\mathbb N$},\end{equation*}
whence, by $c_{A,f}<0$ and (\ref{L}),
\begin{equation}\label{ext-pr'}L:=\lim_{|x|\to\infty, \ x\in A}\,U^\omega(x)>0,
\end{equation}
which however contradicts the assumption $L=0$.

\subsection{Proof of Corollary~\ref{th-main2'}} Since, by assumption, the set $A$ is not inner $\alpha$-thin at infinity, $\omega(\mathbb R^n)>1$ holds if and only if $\omega^A(\mathbb R^n)>1$ (Theorem~\ref{th-eq}(vi$_3$)). Therefore, according to what has just been shown in Section~\ref{sec-prpr}, (\ref{ext-pr'}) must be true, which however is impossible on account of Theorem~\ref{th-eq}(ii$_3$).

\subsection{Proof of Theorem~\ref{shrpp}} Let $A$ be as indicated at the beginning of Section~\ref{sec-descrrr}, and let $A$ be $\alpha$-thin at infinity. By Theorem~\ref{th-eq}(i$_3$), then there is the equilibrium measure $\gamma_A$, introduced in Definition~\ref{def-eq}. Theorem~\ref{th-eq}(ii$_3$), applied to $\gamma_A$ and each of the sets $A^c\cap\{|x|>r\}$, $r\in(0,\infty)$, which are obviously not $\alpha$-thin at infinity, gives
\begin{equation}\label{upto}
\liminf_{A^c\ni x\to\infty_{\mathbb R^n}}\,U^{\gamma_A}(x)=0,
\end{equation}
$\infty_{\mathbb R^n}$ being the Alexandroff point of $\mathbb R^n$. Noting from \cite[Theorem~2.2]{Z-bal2} that
\begin{equation}\label{jj}\varepsilon_x^A(\mathbb R^n)=U^{\gamma_A}(x)\quad\text{for all $x\in\mathbb R^n$},\end{equation}
we thus conclude from (\ref{upto}) that there is a sequence $(x_j)\subset A^c$ convergent to $\infty_{\mathbb R^n}$, and such that
\begin{equation}\label{j}
\lim_{j\to\infty}\,\varepsilon_{x_j}^A(\mathbb R^n)=0.
\end{equation}
Define
\[q_j:=1/\varepsilon_{x_j}^A(\mathbb R^n),\quad\omega_j:=q_j\varepsilon_{x_j},\quad f_j:=-U^{\omega_j};\]
then $1\leqslant q_j<\infty$ (see (\ref{ineqne}) and (\ref{jj})), $q_j\to\infty$ as $j\to\infty$ (see (\ref{j})), and also
\begin{equation}\label{jjj}\omega_j^A\in\mathcal E^+,\quad\omega_j^A(\mathbb R^n)=1.\end{equation}
Applying Theorem~\ref{equality}, we now infer from (\ref{jjj}) that $\lambda_{A,f_j}$ does exist, and moreover
\[\lambda_{A,f_j}=\omega_j^A.\]
Therefore, by virtue of \cite[Theorem~4.1]{Z-bal2} (cf.\ Section~\ref{sec-pr-th-main3} above), $S(\lambda_{A,f_j})$ is given by the right-hand side in (\ref{Dm}), and hence the sequences $(x_j)$ and $(q_j)$ are as required.

\section{Applications}\label{sec-ap}

Assuming again that $A\subset\mathbb R^n$ is arbitrary (not necessarily satisfying $(\mathcal H_1)$), in the following Theorem~\ref{th-ult} we shall be concerned with $z\in\overline{A}$ whose inner $\alpha$-harmonic measure $\varepsilon_z^A$ is of finite energy. Clearly, to this end we only need to examine $z\in A_I$.

\begin{theorem}\label{th-ult}
Given $A\subset\mathbb R^n$ and $z\in A_I$, the following {\rm(i$_4$)}--{\rm(iv$_4$)} are equivalent.
\begin{itemize}
\item[{\rm(i$_4$)}] $\varepsilon_z^A\in\mathcal E^+$.
\item[{\rm(ii$_4$)}] If $A_j:=A\cap\{x\in\mathbb R^n:\ r^{j+1}\leqslant|x-y|<r^j\}$, where $r\in(0,1)$, then
      \begin{equation}\label{iii'}\sum_{j\in\mathbb N}\,\frac{c_*(A_j)}{r^{2j(n-\alpha)}}<\infty.\end{equation}
\item[{\rm(iii$_4$)}] $c_*(A_z^*)<\infty$, where $A_z^*$ is the inversion of $A$ with respect to the sphere $S(z,1)$.
\item[{\rm(iv$_4$)}] $A_z^*$ is inner $\alpha$-ultrathin at infinity {\rm(for definition see \cite[Section~2.3]{Z-bal2})}.
\end{itemize}

Furthermore, if any one of these {\rm(i$_4$)}--{\rm(iv$_4$)} is fulfilled, then the inner $\alpha$-harmonic measure $\varepsilon_z^A$ is representable in the form
\[\varepsilon_z^A=(\gamma_{A^*_z})^*,\]
$(\gamma_{A^*_z})^*$ being the Kelvin transform of the inner equilibrium measure $\gamma_{A^*_z}\in\mathcal E^+$ with respect to the sphere $S_{z,1}$.
\end{theorem}

\begin{proof}
This follows from (\ref{har-eq}) and \cite[Theorem~2.3]{Z-bal2} with $A$ and $A_y^*$ interchanged.
\end{proof}

If $A_I^\circ$ denotes the set of all $z\in A_I$ satisfying the above (i$_4$)--(iv$_4$), then, in general,
\[A_I^\circ\Subset A_I,\]
which is obviously seen by comparing (\ref{iii'}) with (\ref{w}). It is also worth noting that, if $\alpha=2$ while $A$ is Borel, then $z\in A_I^\circ$ if and only if $A^*_z$ is outer $2$-thin at infinity in the sense of Brelot \cite[p.~313]{Brelot}, cf.\ \cite[Remark~2.3]{Z-bal2}.

\begin{example}\label{ex} Let $n=3$ and $\alpha=2$. Consider the rotation body
\begin{equation*}\Delta:=\bigl\{x\in\mathbb R^3: \ 0\leqslant x_1\leqslant1, \
x_2^2+x_3^2\leqslant\exp(-2x_1^{-\varrho})\bigr\},\quad\varrho\in(0,\infty).\end{equation*}
For any $\varrho>0$, the origin $x=0$ is $2$-irregular for $\Delta$ (Lebesgue cusp), and moreover
\[0\in A_I^\circ\iff\varrho>1.\]
(Compare with \cite[Section~V.1.3, Example]{L} and \cite[Example~2.1]{Z-bal2}.)
\end{example}

\begin{theorem}\label{th-ap}
 Assume that $A$ is not inner $\alpha$-thin at infinity and satisfies $(\mathcal H_1)$. If $f$, the external field in Problem~{\rm\ref{pr}}, has the form
 \[f:=f_{q,z}:=-qU^{\varepsilon_z},\]
 where $z\in A_I^\circ\cup\overline{A}^c$ and $q\in(0,\infty)$, then the following {\rm(a$_1$)}--{\rm(e$_1$)} hold true.
 \begin{itemize}
\item[{\rm(a$_1$)}] $\lambda_{A,f_{q,z}}$ exists if and only if $q\geqslant1$.
\item[{\rm(b$_1$)}] $\lambda_{A,f_{q,z}}=q\varepsilon_z^A$ if and only if $q=1$.
\item[{\rm(c$_1$)}] $c_{A,f_{q,z}}=0$ if $q=1$, and $c_{A,f_{q,z}}<0$ otherwise.
\item[{\rm(d$_1$)}] $S(\lambda_{A,f_{q,z}})$ is compact whenever $q>1$.
\item[{\rm(e$_1$)}] Assume that $A$ is closed, coincides with its reduced kernel, and that $A^c$ is connected unless $\alpha<2$. If moreover $q=1$, then
\[S(\lambda_{A,f_{q,z}})=\left\{
\begin{array}{cl}A&\text{if \ $\alpha<2$},\\
\partial A&\text{otherwise}.\\ \end{array} \right.
\]
 \end{itemize}
\end{theorem}

\begin{proof} Since $\varepsilon_z^A\in\mathcal E^+$, the results presented in Section~\ref{sec-main} are applicable to $f:=f_{q,z}$. Therefore, $\lambda_{A,f_{q,z}}$ exists if and only if $q\geqslant1$ (Corollary~\ref{cornoth}), and it is of compact support provided that $q>1$ (Corollary~\ref{th-main2'}). This proves (a$_1$) and (d$_1$).
Next, as $A$ is not $\alpha$-thin at infinity, Theorem~\ref{th-eq}(vi$_3$) applied to $\omega:=q\varepsilon_z$ gives $\omega^A(\mathbb R^n)=q$, which together with (\ref{C2'}) and (\ref{C1'}) yields (c$_1$). If moreover $q=1$, then $\omega^A(\mathbb R^n)=1$, and applying Theorem~\ref{th-main1} shows that, actually, $\lambda_{A,f_{q,z}}=\varepsilon_z^A$ (see the latter relation in (\ref{RRR})), which establishes the "if" part of (b$_1$). To verify the "only if" part, assume that $\lambda_{A,f_{q,z}}=q\varepsilon_z^A$, but, to the contrary, $q\ne1$. Then, again by Theorem~\ref{th-eq}(vi$_3$),
\[1=\lambda_{A,f_{q,z}}(\mathbb R^n)=q\varepsilon_z^A(\mathbb R^n)=q\ne1,\] which is impossible. Finally, under the assumptions of (e$_1$), we have $\lambda_{A,f_{q,z}}=\varepsilon_z^A$ according to (b$_1$), and the claim follows at once from (\ref{Dm}) with $\omega_1:=\varepsilon_z$.
\end{proof}

\begin{example}\label{ex2}Let $n=3$, $\alpha=2$, and let $A:=\Delta\cup H$, where $\Delta$ is introduced in Example~\ref{ex} with $\varrho>1$, whereas $H:=\{x_1\geqslant 1\}$. Then the (closed) set $A$ is not $2$-thin at infinity, while $A_I^\circ$ is reduced to the origin $\{0\}$. Define $f_q=-qU^{\varepsilon_0}$, where $q\in(0,\infty)$ and $\varepsilon_0$ denotes the unit Dirac measure at $x=0$. Then the solution $\lambda_{A,f_q}$ does exist for any $q\geqslant1$, and moreover $S(\lambda_{A,f_1})=\partial A$, whereas $S(\lambda_{A,f_q})$ is a compact subset of $A$ for any $q>1$. See Theorem~\ref{th-ap}, (a$_1$), (e$_1$), and (d$_1$), respectively.
\end{example}

\begin{remark}Under the assumptions of Example~\ref{ex2}, there is no compensation effect between
$-q\varepsilon_0$, where $q\geqslant1$, and $\lambda_{A,f_q}$, though these two charges are oppositely charged and carried by the same conductor $A$. (Actually, this is also the case even when $q$, the attractive mass placed at $0$, becomes arbitrarily large.) This phenomenon seems to be quite surprising and even to contradict our physical intuition. Therefore it would be interesting to illustrate this fact by means of numerical experiments which must of course depend on the geometry of $A$ in a neighborhood of $0\in A_I^\circ$.
\end{remark}

\begin{remark} Assume that $A$ is as indicated in Theorem~\ref{th-ap}(e$_1$), and that $\partial A$ is unbounded unless $\alpha<2$. Another interesting phenomenon is that then, the support of $\lambda_{A,f_{1,z}}$, where $z\in A_I^\circ\cup A^c$, is noncompact, whereas that of $\lambda_{A,f_{1+\beta,z}}$ is already compact for any $\beta>0$. Thus any increase (even arbitrarily small) of the attractive mass $1$ placed at $z\in A_I^\circ\cup A^c$ changes drastically the solution to the problem in question, making it zero in some neighborhood of the point at infinity.
\end{remark}

\begin{remark}
The results of this section can be generalized to any set $A$~--- independently of how large it is at infinity. However, then $1$, the critical value of the attractive mass placed at $z\in A_I^\circ\cup\overline{A}^c$, must be replaced by $1/\varepsilon_z^A(\mathbb R^n)\in[1,\infty)$, the inverse of the total mass of the inner $\alpha$-harmonic measure $\varepsilon_z^A$ of $A$ at the point $z$.
\end{remark}

\section{Acknowledgements} This work was supported by a grant from the Simons Foundation (1030291, N.V.Z.).

\end{document}